\documentclass[11pt]{amsart}
\usepackage[all]{xypic}
\usepackage{latexsym}
\usepackage{amssymb}
\usepackage{amsfonts}
\usepackage{amscd}
\usepackage{amsmath,amsthm}
\usepackage{mathrsfs}

\newtheorem{lemma}{Lemma}[section]

\newtheorem{Prop}[lemma]{Proposition}
\newtheorem{Thm}[lemma]{Theorem}

{

}
\theoremstyle{definition}

\newtheorem{example}[lemma]{Example}

\theoremstyle{remark}

\numberwithin{equation}{section}

\def\bmat{\begin{pmatrix}}
\def\emat{\end{pmatrix}}
\def\-{\smallsetminus}

\def\~{\widetilde}

\def\phi{\varphi}

\def\deg{\text{deg }}

\def\<{\langle}
\def\>{\rangle}

\def\Gr{\operatorname {Gr}}

\def\GrRep{\operatorname {GrRep}}
\def\Fdim{\operatorname {Fdim}}

\def\QGr{\operatorname {QGr}}

\def\id{\operatorname {id}}

\def\Proj{\operatorname {Proj}}

\def\Qcoh{\operatorname {Qcoh}}

\def\Ker{\operatorname {Ker}}
\def\Coker{\operatorname {Coker}}

\def\11{\text{\bf 1}}

\def\NN{{\mathbb N}}

\def\ZZ{{\mathbb Z}}

\def\sM{{\mathscr M}}
\def\sN{{\mathscr N}}

\title[Weighted path algebras]{Category equivalences involving graded modules over quotients of weighted path algebras}

\begin{document}

\author{Cody Holdaway}

\address{Department of Mathematics, Box 354350, Univ.
Washington, Seattle, WA 98195}

\email{codyh3@math.washington.edu}



\begin{abstract}
Let $k$ be a field, $Q$ a finite directed graph, and $kQ$ its path algebra. Make $kQ$ an $\NN$-graded algebra by assigning each arrow a positive degree. Let $I$ be a homogeneous ideal in $kQ$ and write $A=kQ/I$. Let $\QGr A$ denote the quotient of the category of graded right $A$-modules modulo the Serre subcategory consisting of those graded modules that are the sum of their finite dimensional submodules. This paper shows there is a finite directed graph $Q'$ with all its arrows placed in degree 1 and a homogeneous ideal $I'\subset kQ'$ such that $\QGr A \equiv \QGr kQ'/I'$. This is an extension of a result obtained by the author and Gautam Sisodia in \cite{HG}.
\end{abstract}

\maketitle

\pagenumbering{arabic}

\setcounter{section}{0}

\section{Introduction}

\subsection{} In noncommutative projective geometry, there seems to be a consensus that being generated in degree $1$ is ``good.'' 

For example, consider Serre's Theorem: If $A$ is a locally finite commutative graded $k$-algebra generated in degree $1$, then $\QGr A \equiv \Qcoh(\Proj A)$. Serre's Theorem can fail if the algebra is not generated in degree $1$, a counterexample being the polynomial algebra $k[x,y]$ with $\deg x=1$ and $\deg y=2$.  

Another nice theorem that uses generation in degree $1$ is Verevkin's result about the equivalence
$$
\QGr A \equiv \QGr A^{(d)}
$$
where $A^{(d)}$ is the $d$-th Veronese subalgebra of $A$ \cite{ABV}.

Given a graded algebra $A$, is it possible to find a graded algebra $A'$ generated in degree one such that 
$$
\QGr A \equiv \QGr A'?
$$
In \cite{HG} it was shown that the answer is yes when $A$ is a path algebra or a monomial algebra. This article extends these results to include the case where $A$ is any quotient of a path algebra by a finitely generated homogeneous ideal.

Lets consider the example with the commutative polynomial algebra $A=k[x,y]$ where $\deg x=1$ and $\deg y=2$. $A$ is the quotient of the path algebra $kQ$ modulo the ideal $I=(xy-yx)$ where $Q$ is the quiver
$$
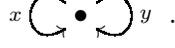
\UseComputerModernTips
\xymatrix{ {\bullet} \ar@(ul,dl)_{x} \ar@(ur,dr)^{y} & {.} }
$$

Let $Q'$ be the quiver:
$$
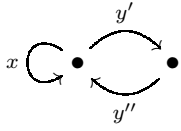
\UseComputerModernTips
\xymatrix{ {\bullet} \ar@(ul,dl)_{x} \ar@/^1pc/[r]^{y'} & {\bullet} \ar@/^1pc/[l]^{y''}}
$$
and give $kQ'$ the grading where all arrows have degree $1$. It is shown in \cite{HG} that $\QGr kQ \equiv \QGr kQ'$. That is, the noncommutative projective schemes $\Proj_{nc}kQ$ and $\Proj_{nc}kQ'$ are isomorphic.

The scheme $\Proj_{nc}k[x,y]$ is a ``closed subsecheme'' of $\Proj_{nc}kQ$ defined by the ideal $I=(xy-yx)$. Since $\Proj_{nc}kQ \cong \Proj_{nc}kQ'$, the space $\Proj_{nc}k[x,y]$ should correspond to some ``closed subscheme'' of $\Proj_{nc}kQ'$. 

One guess might be that $\Proj_{nc}k[x,y]$ corresponds to the closed subscheme of $\Proj_{nc}kQ'$ cut out by the ideal $I'=(xy'y''-y'y''x)$. The methods of this paper show this is true. More explicitly, the main result shows
$$
\QGr k[x,y]\equiv \QGr kQ'/I'.
$$ 

This equivalence is rather interesting. The algebra $k[x,y]$ is a connected Noetherian domain while $kQ'/I'$ is none of these. However, $kQ'/I'$ is generated in degree $1$. Thus, in trying to understand $\QGr k[x,y]$, one can use whichever algebra is most suited to the question at hand.

The principal result of this paper is:

\begin{Thm} \label{thm.main}
Let $Q$ be a weighted quiver and $I$ a finitely generated homogeneous ideal in $kQ$. There is a quiver $Q'$ with all arrows having degree $1$, a finitely generated homogeneous ideal $I'\subset kQ'$, and an equivalence of categories 
$$
F:\QGr kQ/I \equiv \QGr kQ'/I'
$$
which respects shifting. That is, $F(\sM(1))\cong F(\sM)(1)$ for all $\sM \in \QGr kQ/I$.
\end{Thm}

\subsection{Notation and definitions} \label{sec.def} 
Throughout, $Q=(Q_0,Q_1,s,t)$ will always denote a finite quiver, i.e., a finite directed graph. The set $Q_0$ is called the vertex set, $Q_1$ the arrow set and $s,t:Q_1\to Q_0$ will be the source and target maps respectively. Given a field $k$, the path algebra $kQ$ is the algebra with basis consisting of all paths in $Q$, including a trivial path $e_v$ at each vertex $v$.

Given two paths $p=a_1\cdots a_n$ and $q=b_1\cdots b_m$, the product $pq$ is the path $a_1\cdots a_nb_1\cdots b_m$ if $t(a_n)=s(b_1)$ and is zero otherwise.

Call the pair $(Q,\deg)$ a {\it weighted quiver} if $Q$ is a finite quiver and $\deg:Q_1 \to \NN_{>0}$. Usually, the $\deg$ part of the notation $(Q,\deg)$ will be dropped.

A weighted quiver determines an $\NN$-graded path algebra $kQ$ where the degree of the arrow $a$ is $\deg(a)$ and the trivial paths have degree zero. The term {\it weighted path algebra} will mean the path algebra of a weighted quiver. The term path algebra will always mean the arrows have degree $1$. 

Given an $\NN$-graded $k$-algebra $A$, $\Gr A$ will denote the category of $\ZZ$-graded right $A$ modules with degree preserving homomorphisms. $\Fdim A$ will denote the localizing subcategory of $\Gr A$ consisting of all graded modules which are the sum of their finite-dimensional submodules. The quotient of $\Gr A$ by $\Fdim A$ is denoted $\QGr A$ and the canonical quotient functor will be denoted
$$
\pi^*: \Gr A \to \QGr A.
$$

The functor $\pi^*$ is exact and the subcategory $\Fdim A$ is localizing, that is, $\pi^*$ has a right adjoint which will be denoted $\pi_*$. 

\section{The category of graded representations with relations.}

Associated to a weighted quiver $Q$ is the category of graded representations $\GrRep Q$. A graded representation is the data $M=(M_v,M_a)$ where for each vertex $v$, $M_v$ is a $\ZZ$-graded vector space over $k$ ($k$ is in degree zero) and for each arrow $a$, $M_a:M_{s(a)}\to M_{t(a)}$ is a degree $\deg(a)$ linear map.

A morphism $\phi:M \to N$ is a collection of degree $0$ linear maps $\phi_v:M_v\to N_v$ for each vertex $v$ such that for each arrow $a\in Q_1$, the diagram
$$
\UseComputerModernTips
\xymatrix{ {M_{s(a)}} \ar[r]^{M_a} \ar[d]_{\phi_{s(a)}} & {M_{t(a)}} \ar[d]^{\phi_{t(a)}}\\
           {N_{s(a)}} \ar[r]_{N_a} & {N_{t(a)}} }
$$
commutes.

The categories $\Gr kQ$ and $\GrRep Q$ are equivalent. An explicit equivalence is given by sending a graded module $M$ to the data $(Me_v,M_a)$ where $M_a:Me_{s(a)}\to Me_{t(a)}$ is the degree $\deg(a)$ linear map induced by the action of $a$. 

If $p=a_1\cdots a_m$ is a path in $Q$, then given any graded representation $(M_v,M_a)$, $p$ determines a degree $\deg(p)$ linear map $M_p:M_{s(a_1)}\to M_{t(a_m)}$ which is the composition
$$
M_p=M_{a_m}\circ \cdots \circ M_{a_1}.
$$
Given a linear combination $\rho=\sum \alpha_ip_i$, where $\alpha_i\in k$ and the $p_i$ are paths in $Q$ with the same source and target, we get a linear map 
$$
M_{\rho}=\sum \alpha_iM_{p_i}.
$$

Let $A=kQ/I$ be a weighted path algebra modulo an ideal $I$ generated by a finite number of homogeneous elements. Because of the idempotents $e_v$, we can write 
$$
I=(\rho_1,\ldots,\rho_n)
$$
where $\rho_i$ is a linear combination of paths of the same degree all of which have the same source and target.

Let $\GrRep (Q,\rho_1,\ldots,\rho_n)$ denote the full subcategory of $\GrRep Q$ consisting of all the graded representations $(M_v,M_a)$ such that $M_{\rho_i}=0$ for all $i=1,\ldots,n$. The equivalence $\Gr kQ \equiv \GrRep Q$ induces an equivalence $\Gr kQ/I \equiv \GrRep (Q,\rho_1,\ldots,\rho_n)$. From now on, the categories $\Gr kQ/I$ and $\GrRep (Q,\rho_1,\ldots,\rho_n)$ will be identified. 

\section{Proof of Theorem \ref{thm.main}} 

\subsection{} \label{sec.secmain}

The proof of Theorem \ref{thm.main} follows section 3 in \cite{HG} very closely. The details, with the appropriate modifications for the more general case, are reproduced here for convenience of the reader. 

Given a weighted quiver $Q$, define the {\it weight discrepancy} to be the non negative integer
$$
D(Q):=\left(\sum_{a\in Q_1}\deg(a)\right)-|Q_1|.
$$ 
Note that $D(Q)=0$ if and only if each arrow in $Q$ has degree $1$. The proof of Theorem \ref{thm.main} will be based on induction on $D(Q)$.

Let $Q$ be a weighted quiver and suppose $b$ is an arrow with $\deg(b)>1$. Define a new quiver $Q'$ from $Q$ by declaring
\begin{eqnarray*}
Q'_0&:=& Q_0 \sqcup \{z\} \\
Q'_1&:=& (Q_1\-\{b\}) \sqcup \{b':s(b) \to z, b'':z\to t(b)\}.
\end{eqnarray*}
Make $Q'$ a weighted quiver by letting each arrow in $Q'_1\-\{b',b''\}$ have the same degree as it had in $Q_1$ and letting $\deg(b')=1$ and $\deg(b'')=\deg(b)-1$. From the construction of $Q'$ it follows that
$$
D(Q')=D(Q)-1.
$$

\begin{example} \label{ex.ex1}
Let $Q$ be the quiver
$$
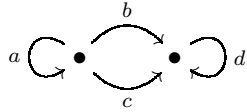
\UseComputerModernTips
\xymatrix{ {\bullet} \ar@(ul,dl)_{a} \ar@/^1pc/[r]^{b} \ar@/_1pc/[r]_{c} & {\bullet} \ar@(ur,dr)^{d} }
$$
with $\deg(b)>1$. The associated quiver $Q'$ is 
$$
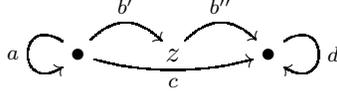
\UseComputerModernTips
\xymatrix{ {\bullet}\ar@(ul,dl)_{a} \ar@/_0.5pc/[rr]_{c} \ar@/^1pc/[r]^{b'} & {z}\ar@/^1pc/[r]^{b''} & {\bullet} \ar@(ur,dr)^{d} }
$$
with $\deg(b')=1$ and $\deg(b'')=\deg(b)-1$.
\end{example}

Let $Q$ be a weighted quiver and $Q'$ the associated quiver constructed above. Given a path $p=a_1\cdots a_m$ in $Q$, let $f(p)$ be the path in $Q'$ which is obtained by replacing every occurrence of $b$ with $b'b''$ while leaving the path unchanged if there is no occurrence of $b$. For the quiver in example \ref{ex.ex1},
$$
f(a^2bd)=a^2b'b''d
$$ 
while
$$
f(acd)=acd.
$$
As $\deg(b'b'')=\deg(b)$, the map $f$ preserves the degree of paths. Hence, $f$ determines a graded $k$-linear map $f:kQ \to kQ'$ which can be seen to respect multiplication.

\subsection{} Let $Q$ be a weighted quiver and $Q'$ the associated quiver as in section \ref{sec.secmain}. Given a graded representation $M \in \Gr kQ$, let $F(M)$ be the following graded representation in $\Gr kQ'$:

For the vertices;
\begin{itemize}
\item $F(M)_v:=M_v$ for all $v\in Q'_1\-\{z\}$,
\item $F(M)_z:=M_{s(b)}(-1)$,
\end{itemize}
while for the arrows;
\begin{itemize}
\item $F(M)_a:=M_a$ for all $a\in Q'_1\-\{b',b''\}$,
\item $F(M)_{b'}:=\id:M_{s(b)} \to M_{s(b)}(-1)$ considered a linear map of degree $1$,
\item $F(M)_{b''}:=M_b:M_{s(b)}(-1) \to M_{t(b)}$ considered a linear map of degree $\deg(b)-1$.
\end{itemize}

Given a morphism $\phi:M \to M'$ in $\Gr kQ$, define $F(\phi):F(M) \to F(M')$ by
\begin{itemize}
\item $F(\phi)_v:=\phi_v$ for all $v\in Q'_0\-\{z\}=Q_0$, and
\item $F(\phi)_z:=\phi_{s(b)}(-1):M_{s(b)}(-1) \to M'_{s(b)}(-1)$.
\end{itemize}

It is shown in \cite{HG} that $F:\Gr kQ \to \Gr kQ'$ is an exact functor for which 
$$
F(M(1))\cong F(M)(1).
$$

Let $p=a_1\cdots a_m$ be a path in $Q$ and $f(p)$ the associated path in $Q'$. From the definition of the functor $F$,
$$
F(M)_{f(p)}=M_p.
$$
To see this, note $f(p)=f(a_1)\cdots f(a_m)$ so 
$$
F(M)_{f(p)}=F(M)_{f(a_m)}\cdots F(M)_{f(a_1)}.
$$ 
If $a_i\neq b$, then $f(a_i)=a_i$ and thus $F(M)_{f(a_i)}=F(M)_{a_i}=M_{a_i}$. If $a_i=b$, then $f(a_i)=b'b''$ and thus $F(M)_{f(a_i)}=F(M)_{b'b''}=F(M)_{b''}F(M)_{b'}=M_b\circ \id=M_b$. Hence, if $\rho=\sum \alpha_ip_i$ is a linear combination of paths with the same source and target, then
$$
F(M)_{f(\rho)}=\sum \alpha_iF(M)_{f(p_i)}=\sum \alpha_iM_{p_i}=M_{\rho}.
$$

Let $I=(\rho_1,\ldots,\rho_n)\subset kQ$ be a homogeneous ideal. As before, 
$$
\rho_i=\sum_{j=1}^m \alpha_jp_j
$$ 
is a linear combination of paths of the same degree such that $s(p_j)=s(p_{j'})$ and $t(p_j)=t(p_{j'})$ for all pairs $(j,j')$. 

Suppose $M\in \Gr kQ/I$. For all $\rho_i\in I$, $M_{\rho_i}=0$. Hence, for the representation $F(M)$, $F(M)_{f(\rho_i)}=M_{\rho_i}=0$ which implies $F(M) \in \Gr kQ'/I'$ where $I'$ is the ideal 
$$
I'=(f(\rho_1),\ldots,f(\rho_n)).
$$ 
Therefore, the functor $F:\Gr kQ \to \Gr kQ'$ induces a functor $F:\Gr kQ/I \to \Gr kQ'/I'$.

Let $N$ be a representation of $kQ'$. Define $G(N)$ to be the following representation of $kQ$:

For the vertices,
\begin{itemize}
\item $G(N)_v:=N_v$ for all vertices $v\in Q_0=Q'_0\-\{z\}$,
\end{itemize}
while for the arrows
\begin{itemize}
\item $G(N)_a:=N_a$ for all $a\in Q_1\-\{b\}$, and
\item $G(N)_b:=N_{b''}\circ N_{b'}$ which is a linear map of degree $\deg(b''b')=\deg(b)$. 
\end{itemize}

Given a morphism $\psi:N \to N'$ in $\Gr kQ'$, define $G(\psi):G(N) \to G(N')$ by
\begin{itemize}
\item $G(\psi)_v:=\psi_v$ for all $v\in Q_0=Q'_0\-\{z\}$.
\end{itemize}

$G$ is a functor $\Gr kQ' \to \Gr kQ$.

Let $N$ be a representation in $\Gr kQ'$ and $p=a_1\cdots a_m$ a path in $Q$. Since $G(N)_b=N_{b''}N_{b'}$ and $G(N)_a=N_a$ for $a\in Q_1\-\{b\}$, it follows that
$$
G(N)_p=N_{f(p)}
$$
and more generally,
$$
G(N)_{\rho}=N_{f(\rho)}
$$
for any linear combination of paths with the same source and target. Hence, if $N$ is a representation in $\Gr kQ'/I'$, then for all $\rho_i\in I$,
$$
G(N)_{\rho_i}=N_{f(\rho_i)}=0.
$$
Hence, the functor $G:\Gr kQ' \to \Gr kQ$ induces a functor $G:\Gr kQ'/I' \to \Gr kQ/I$.

From the definitions of $F$ and $G$, it can be seen that $GF=\id_{\Gr kQ/I}$.

Let $N\in \Gr kQ'/I'$, then the module $FG(N)$ is given by the data
\begin{itemize}
\item $FG(N)_v=N_v$ for $v\in Q'_0\-\{z\}$,
\item $FG(N)_z=N_{s(b)}(-1)$,
\item $FG(N)_a=N_a$ for all $a\in Q'_1\-\{b',b''\}$,
\item $FG(N)_{b'}=\id:N_{s(b)} \to N_{s(b)}(-1)$ considered a degree one linear map,
\item $FG(N)_{b''}=N_{b''}\circ N_{b'}:N_{s(b)}(-1) \to N_{t(b)}$.
\end{itemize}

For each $N\in \Gr kQ'/I'$, define $\epsilon_N:FG(N) \to N$ by $(\epsilon_N)_v=\id$ for $v\neq z$ and $(\epsilon_N)_z=N_{b'}$ considered as a degree zero map from $FG(N)_z=N_{s(b)}(-1) \to N_z$.

\begin{Prop}
The assignment $N\mapsto \epsilon_N$ is a natural transformation $\epsilon:FG \to \id_{\Gr kQ'/I'}$. Let $\eta:\id_{\Gr kQ/I} \to GF$ be the identity natural transformation. Then $F$ is left adjoint to $G$ with unit $\eta$ and counit $\epsilon$.
\end{Prop}
\begin{proof}
See Propositions $3.3$ and $3.4$ in \cite{HG}.
\end{proof}

\subsection{} 

Let $\pi^*:\Gr kQ'/I' \to \QGr kQ'/I'$ be the canonical quotient functor and $\pi_*$ it's right adjoint. Let $\sigma:\id_{\Gr kQ'/I'} \to \pi_*\pi^*$ be the unit and $\tau:\pi^*\pi_* \to \id_{\QGr kQ'/I'}$ the counit of the adjoint pair $(\pi^*,\pi_*)$. Using the adjoint pair $(F,G)$, we get the adjoint pair $(\pi^*F,G\pi_*)$ where
\begin{itemize}
\item $G\sigma F\cdot \eta:\id_{\Gr kQ/I}\to G\pi_*\circ \pi^*F$ is the unit and
\item $\tau \cdot \pi^*\epsilon \pi_*:\pi^*F \circ G\pi_* \to \id_{\QGr kQ'/I'}$ is the counit.
\end{itemize} 

As $\pi^*$ and $F$ are exact so is $\pi^*F$.

\begin{lemma}
The kernel of $\pi^*F:\Gr kQ/I \to \QGr kQ'/I'$ is 
$$
\Ker \pi^*F=\Fdim kQ/I.
$$ 
\end{lemma}
\begin{proof}
Same as the proof of Lemma $3.5$ in \cite{HG}. 
\end{proof}

\begin{Prop} \label{prop.prop2}
For every module $N\in \Gr kQ'/I'$, $\pi^*(\epsilon_N)$ is an isomorphism.
\end{Prop}
\begin{proof}
For each vertex $v\in Q'_0\-\{z\}$, $\epsilon_N=\id_{N_v}$. Hence, $(\Ker \epsilon_N)_v$ and $(\Coker \epsilon_N)_v$ are zero for all vertices $v\in Q'_0\-\{z\}$. Hence, the modules $\Ker \epsilon_N$ and $\Coker \epsilon_N$ are supported only on the vertex $z$. Thus, every arrow acts trivially on $\Ker \epsilon_N$ and $\Coker \epsilon_N$ showing they are both in $\Fdim kQ'/I'$. Hence, the map $\pi^*(\epsilon_N)$ is an isomorphism.
\end{proof}

\begin{Thm} \label{thm.main2}
The functor $\pi^*F:\Gr kQ/I \to \QGr kQ'/I'$ induces an equivalence of categories
$$
\QGr kQ/I \equiv \QGr kQ'/I'.
$$
\end{Thm}
\begin{proof}
As $F$ and $\pi^*$ preserve shifting, $\pi^*F$ preserves shifting.
The functor $\pi^*F$ is an exact functor with a right adjoint $G\pi_*$. For every object $\sN \in \QGr kQ'/I'$, the map $\pi^*(\epsilon_{\pi_*\sN})$ is an isomorphism by Proposition \ref{prop.prop2}. By \cite[Prop. 4.3, pg. 176]{Pop}, the counit $\tau$ of the adjoint pair $(\pi^*,\pi_*)$ is a natural isomorphism. Hence, the counit $\tau \cdot \pi^*\epsilon \pi_*$ is a natural isomorphism as
$$
(\tau \cdot \pi^*\epsilon \pi_*)_{\sN}=\tau_{\sN}\circ \pi^*(\epsilon_{\pi_*\sN})
$$ 
is a composition of isomorphisms for all $\sN \in \QGr kQ'/I'$.

Thus, the right adjoint $G\pi_*$ is fully faithful. By \cite[Theorem 4.9, pg. 180]{Pop}, $\pi^*F$ induces an equivalence
$$
\frac{\Gr kQ/I}{\Ker \pi^*F} \equiv \QGr kQ'/I'
$$
which preserves shifting. As $\Ker \pi^*F=\Fdim kQ/I$, the Theorem is proved.
\end{proof}

\subsection{Proof of Theorem \ref{thm.main}.}

The proof of Theorem \ref{thm.main} now follows by induction on the weight discrepancy. If $kQ/I$ is a quotient of a weighted path algebra for which $D(Q)=0$, then every arrow in $Q$ has degree $1$ and there is nothing to prove. Suppose $D(Q)>1$ and let $b$ be an arrow in $Q$ of degree greater than $1$. Let $Q'$ be the quiver obtained from $Q$ by replacing the arrow $b$ with two arrows as in Section \ref{sec.secmain} and $I'$ the ideal obtained from the ideal $I$. By Theorem \ref{thm.main2} there is an equivalence
$$
\QGr kQ/I \equiv \QGr kQ'/I'.
$$
which respects shifting. Since $D(Q')=D(Q)-1$, we can find, by induction, a quiver $Q''$ with all arrows in degree $1$ and a homogeneous ideal $I''\subset kQ''$ such that 
$$
\QGr kQ'/I' \equiv \QGr kQ''/I''
$$
where the equivalence respects shifting. Hence, $\QGr kQ/I \equiv \QGr kQ''/I''$ via an equivalence which respects shifting. 


\end{document}